\documentclass[11pt,a4paper]{article}

\usepackage[margin=23mm]{geometry}
\usepackage{amsmath,amssymb,amsthm,mathtools}
\usepackage{enumitem}
\usepackage{hyperref}
\hypersetup{
	pdftitle={A Near-Optimal Linear Range for the Erdos Matching Conjecture},
	pdfauthor={Mengyu Cao, Hong Liu, and Haixiang Zhang},
	pdfsubject={A near-optimal linear range for the Erdos matching conjecture},
	pdfkeywords={Erdos matching conjecture, uniform hypergraph, matching, stability, fractional vertex cover, small deviations},
	colorlinks=true,
	linkcolor=blue,
	citecolor=blue,
	urlcolor=blue
}
\usepackage{microtype}
\usepackage{authblk}
\usepackage{aliascnt}
\usepackage[noabbrev]{cleveref}
\allowdisplaybreaks

\newcommand{\R}{\mathbb R}
\newcommand{\E}{\mathbb E}
\newcommand{\Prob}{\mathbb P}
\newcommand{\1}{\mathbf 1}
\newcommand{\cA}{\mathcal A}
\newcommand{\cB}{\mathcal B}
\newcommand{\cF}{\mathcal F}
\newcommand{\cG}{\mathcal G}
\newcommand{\cH}{\mathcal H}
\newcommand{\cJ}{\mathcal J}
\newcommand{\cM}{\mathcal M}
\newcommand{\eps}{\varepsilon}
\newcommand{\dd}{\,\mathrm d}
\newcommand{\Dir}{\operatorname{Dir}}
\newcommand{\vol}{\operatorname{vol}}

\newtheorem{theorem}{Theorem}[section]

\newaliascnt{lemma}{theorem}
\newtheorem{lemma}[lemma]{Lemma}
\aliascntresetthe{lemma}

\newaliascnt{corollary}{theorem}
\newtheorem{corollary}[corollary]{Corollary}
\aliascntresetthe{corollary}

\newaliascnt{conjecture}{theorem}
\newtheorem{conjecture}[conjecture]{Conjecture}
\aliascntresetthe{conjecture}

\theoremstyle{remark}
\newaliascnt{remark}{theorem}
\newtheorem{remark}[remark]{Remark}
\aliascntresetthe{remark}

\crefname{theorem}{Theorem}{Theorems}
\Crefname{theorem}{Theorem}{Theorems}
\crefname{lemma}{Lemma}{Lemmas}
\Crefname{lemma}{Lemma}{Lemmas}
\crefname{corollary}{Corollary}{Corollaries}
\Crefname{corollary}{Corollary}{Corollaries}
\crefname{conjecture}{Conjecture}{Conjectures}
\Crefname{conjecture}{Conjecture}{Conjectures}
\crefname{remark}{Remark}{Remarks}
\Crefname{remark}{Remark}{Remarks}

\title{A Near-Optimal Linear Range for the Erd\H{o}s Matching Conjecture}
\author[1]{Mengyu Cao\thanks{E-mail: \texttt{myucao@ruc.edu.cn}. Supported by the National Natural Science Foundation of China (12301431) and Beijing Natural Science Foundation (1262010).}}
\author[2]{Hong Liu\thanks{E-mail: \texttt{hongliu@ibs.re.kr}. Supported by the Institute for Basic Science under grant IBS-R029-C4.}}
\author[3]{Haixiang Zhang\thanks{Corresponding author. E-mail: \texttt{zhang-hx22@mails.tsinghua.edu.cn}.}}

\affil[1]{\small Institute for Mathematical Sciences, Renmin University of China, Beijing 100086, China}
\affil[2]{\small Extremal Combinatorics and Probability Group, Institute for Basic Science, Daejeon, South Korea}
\affil[3]{\small Department of Mathematical Sciences, Tsinghua University, Beijing 100084, China}

\date{}

\begin{document}
	\maketitle
	
		\begin{abstract}
		The Erd\H{o}s Matching Conjecture is governed by two competing ways of
		excluding $s+1$ disjoint edges: one may concentrate all edges on fewer than
		$k(s+1)$ vertices, or force every edge to meet a fixed $s$-set.  We
		determine a near-optimal range in which the second construction is extremal.
		For every fixed $k\ge2$, there is $s_0(k)$ such that, whenever $s\ge s_0(k)$
		and $n\ge(k+1)s$, every $\cF\subseteq\binom{[n]}k$ with $\nu(\cF)\le s$
		satisfies
		\[
		|\cF|\le\binom nk-\binom{n-s}k,
		\]
		with equality only for the family of all $k$-sets meeting a fixed $s$-set.
		This improves the best previous general linear coefficient
		from $(5k-2)/3$ to $k+1$.  In particular, the parameterized form of our argument further lowers the
		coefficient to $k+0.6$ for $k\ge5$.  Since the two conjectured constructions exchange asymptotic
		dominance at $n=(\rho_k+o(1))s$ for a coefficient $\rho_k\in(k,k+1)$, our
		range lies less than one unit above the unavoidable barrier.  We also prove a
		stability theorem showing that cover families are the only near-extremal
		configurations throughout this range.
		A key ingredient in our proof is a probabilistic rigidity statement which forces near-extremal fractional covers
        to be almost integral.
	\end{abstract}
	
	\section{Introduction}
	
	The Erd\H{o}s Matching Conjecture is a prototype for a phase transition between
	distinct extremal mechanisms.  A $k$-uniform hypergraph can avoid $s+1$
	pairwise disjoint edges in two fundamentally different ways: it can concentrate
	all of its edges on too few vertices, or it can impose a small transversal that
	meets every edge.  The first obstruction is local and clique-like, whereas the
	second is global and cover-like.  As the ratio $n/s$ varies, these constructions
	exchange dominance.  Locating that transition requires more than comparing their
	sizes: one must rule out intermediate configurations that behave fractionally
	like a cover without possessing an actual small transversal.  This is why the
	problem lies naturally at the intersection of extremal set theory, linear
	programming duality, and sharp probability inequalities.
	
	For a finite set $V$, write $\binom Vk$ for the collection of its $k$-element
	subsets, and put $[n]=\{1,\ldots,n\}$.  A family
	$\cF\subseteq\binom{[n]}k$ is a $k$-uniform hypergraph, or a $k$-graph.  A
	\emph{matching} is a collection of pairwise disjoint edges, and $\nu(\cF)$
	denotes the maximum size of a matching in $\cF$.  For $U\subseteq[n]$, let
	$\cF[U]=\cF\cap\binom Uk$, and write $\cF\mathbin\triangle\cG$ for symmetric
	difference.
	
	When $n<k(s+1)$, the complete $k$-graph has matching number at most $s$, so the
	problem becomes nontrivial only for $n\ge k(s+1)$.   In that range the
    conjecture compares the following two constructions.  The first is the clique family
	\[
	\cB_k(s):=\binom{[k(s+1)-1]}k.
	\]
	For $S\subseteq V$, the second is the cover family
	\[
	\cA_V(S):=\left\{E\in\binom Vk:E\cap S\ne\varnothing\right\}.
	\]
	When $V=[n]$, we write $\cA(S)=\cA_{[n]}(S)$.  If $|S|=s$, then
$	|\cA(S)|=M_k(n,s):=\binom nk-\binom{n-s}k.$
	Both constructions have matching number at most $s$: the first uses fewer than
	$k(s+1)$ vertices, while pairwise disjoint members of the second must meet $S$
	in distinct vertices.  Erd\H{o}s conjectured in 1965 that one of these two
	constructions is always extremal \cite{Erdos1965}.
	
	\begin{conjecture}\label{conj:emc}
		Let $n,k,s$ be positive integers with $n\ge k(s+1)$.  If
		$\cF\subseteq\binom{[n]}k$ and $\nu(\cF)\le s$, then
		\[
		|\cF|\le
		\max\left\{\binom{k(s+1)-1}k,
		\binom nk-\binom{n-s}k\right\}.
		\]
	\end{conjecture}
	
	Despite its elementary formulation, the conjecture remains open in general for
	$k\ge5$.  The graph case $k=2$ is the Erd\H{o}s--Gallai matching theorem
	\cite{ErdosGallai1959}, and the case $k=3$ was settled through work of Frankl,
	R\"odl and Ruci\'nski and of Frankl
	\cite{FranklRodlRucinski2012,Frankl2017Triple}.  For arbitrary uniformity, results for $n$ sufficiently large relative to $s$ go back to Bollob\'as, Daykin and Erd\H{o}s
	\cite{BollobasDaykinErdos1976}.  On the cover side, Frankl proved the conjectured
	bound $M_k(n,s)$ for $n\ge(2k-1)s+k$ \cite{Frankl2013}; Frankl and Kupavskii
	then reduced the coefficient, for sufficiently large $s$, to
	\[
	 n\ge \frac{5k-2}{3}s.
	\]
	This is the strongest previously known theorem that applies uniformly to every
	fixed $k$ \cite{FranklKupavskii2022}.  Complementary progress closer to the
	clique range was obtained by Frankl and by Kolupaev and Kupavskii
	\cite{Frankl2017Range,KolupaevKupavskii2023}.  In uniformity four, Frankl, Lu,
	Ma and Wu proved the conjecture for $n\ge5s$ in the large-parameter regime
	\cite{FranklLuMaWu2026}, and Hou, Hu and Liu subsequently proved the full
	conjecture for every $s\ge6961$ \cite{HouHuLiu2026}.
	
	Our main theorem reduces the general cover-side coefficient from $(5k-2)/3$ to
	$k+1$.
	
	\begin{theorem}\label{thm:main}
		For every fixed integer $k\ge2$, there exists $s_0=s_0(k)$ such that the following
		holds.  If $s\ge s_0$, $n\ge(k+1)s$, and
		$\cF\subseteq\binom{[n]}k$ satisfies $\nu(\cF)\le s$, then
		\[
		|\cF|\le M_k(n,s)=\binom nk-\binom{n-s}k.
		\]
		Equality holds if and only if $\cF=\cA(S)$ for some
		$S\in\binom{[n]}s$.
	\end{theorem}
	
	\cref{thm:main} reduces the
	additive gap between the known general range and the unavoidable transition from
	order $k$ to less than one, i.e.~the range
	$n\ge(k+1)s$ is near-optimal. Indeed, let $\rho_k$ be the unique real number satisfying
$	\rho_k^k-(\rho_k-1)^k=k^k.$
	The left-hand side is strictly increasing, and its values at $k$ and $k+1$ lie
	respectively below and above $k^k$, so
	\[
	k<\rho_k<k+1.
	\]
	Comparing the leading terms of $|\cB_k(s)|$ and $M_k(n,s)$ shows that the two
	constructions exchange dominance at
$	 n=\rho_k s+o(s).$
	Consequently, no theorem asserting that the cover family is extremal can hold
	asymptotically below the coefficient $\rho_k$.

	\begin{remark}\label{rem:known-cases}
		For $k=2$ and $k=3$, stronger complete results are known
		\cite{ErdosGallai1959,FranklRodlRucinski2012,Frankl2017Triple}.  For $k=4$,
		\cref{thm:main} recovers the $n\ge5s$ conclusion of Frankl, Lu, Ma and Wu,
		whereas the theorem of Hou, Hu and Liu covers the full admissible range for
		$s\ge6961$.  Thus the genuinely new uniformity-specific cases begin at $k=5$,
		although the method is uniform in $k$.
	\end{remark}
	
	The transition coefficient $\rho_k$ also has an exact probabilistic meaning.
	Suppose that $w:[n]\to[0,1]$ is a fractional vertex cover of average weight $x$,
	and let $Z$ be the weight of a uniformly chosen vertex.  The density of the
	$k$-sets certified by $w$ is controlled by
	\[
	\Prob(Z_1+\cdots+Z_k\ge1),
	\]
	where $Z_1,\ldots,Z_k$ are independent copies of $Z$.  The genuine cover
	construction corresponds to the Bernoulli law
	\[
	(1-x)\delta_0+x\delta_1,
	\]
	whose tail probability is
$	q_k(x):=1-(1-x)^k.$
	The clique construction corresponds asymptotically to the fractional model
	\[
	(1-kx)\delta_0+kx\delta_{1/k},
	\]
	whose tail probability is $(kx)^k$.  These two probabilities are equal precisely
	at $x=1/\rho_k$.  Thus the extremal phase transition is already present in the
	one-dimensional optimization problem for the law of a fractional cover.

	The proof can be formulated for any real number $f_k>k$ for which the
	following probabilistic statement holds: for every $0<x\le1/f_k$ and every
	collection of independent, identically distributed random variables
	$Z_1,\ldots,Z_k\in[0,1]$ with
	$\E Z_i\le x$,
	\[
	\Prob(Z_1+\cdots+Z_k\ge1)\le q_k(x),
	\]
	with equality if and only if their common law is
	$\mu_x=(1-x)\delta_0+x\delta_1$.  We prove this statement for
	$f_k=k+1$, while \cref{conj:optimal-tail} would imply it for every
	$f_k>\rho_k$.
	
	This is not merely an analogy.  The assertion that the Bernoulli law maximizes
	the tail probability for $x<1/\rho_k$ is exactly the cover-dominant part of the
	i.i.d. maximal-tail conjecture of \L{}uczak, Mieczkowska and \v{S}ileikis
	\cite{LuczakMieczkowskaSileikis2017}; their full conjecture is equivalent to the
	asymptotic fractional Erd\H{o}s matching conjecture.  We prove the equality-rigid
	statement needed in the range $x\le1/(k+1)$.  As explained in
	\cref{sec:optimal-conjecture}, extending it to every $x<1/\rho_k$ would improve
	\cref{thm:main} to the asymptotically optimal range
	$n\ge(\rho_k+\eps)s$ for every fixed $\eps>0$.

	The fractional-matching and probabilistic components of this approach have
	important antecedents. Alon, Huang and Sudakov
	\cite{AlonHuangSudakov2012} connected fractional covers with
	probabilities, and Alon, Frankl, Huang, R\"odl, Ruci\'nski and Sudakov \cite{AlonEtAl2012} developed this viewpoint systematically for hypergraph matching thresholds. In their minimum-degree setting, they transferred asymptotic fractional thresholds to integral ones by random sampling and sparsification.  They also made explicit the reduction relevant here: if an
	$k$-graph has fractional matching number less than $xm$, linear-programming duality supplies a fractional vertex cover $w$ of total weight less than $xm$; for the empirical law of the weights $w(v)$, the probability $\Prob(Z_1+\cdots+Z_\ell<1)$ measures, up to the negligible error, a family of nonedges. Thus a lower bound for this probability gives an upper bound for the hypergraph.  The direct precursor to our stability framework is the recent work of Frankl, Lu, Ma and Wu \cite{FranklLuMaWu2026}.  For $4$-graphs they first proved a fractional stability theorem, converted it into stability for ordinary matchings using the sampling mechanism of \cite{AlonEtAl2012}, and then upgraded
	stability to the exact Erd\H{o}s matching conjecture for $n\ge5s$ by a local matching argument.  We follow this broad philosophy, but provide a parameterized implementation for every fixed $k$: weak regularity turns the ordinary matching constraint into an approximate fractional cover, an inequality and compactness force its empirical law to approach the Bernoulli cover law, and an equality-sensitive local argument gives the unique exact extremal family.  Moreover, the parameter $f_k$ is preserved throughout the proof, so any improvement of the probability inequality transfers directly to the same Erd\H{o}s matching conjecture coefficient.
	
	A sharp density estimate alone is insufficient for the exact theorem.  We need
	to know that a family whose size is close to $M_k(n,s)$ actually resembles the
	cover construction.  The structural core of the paper is therefore the following
	stability theorem, uniform throughout every compact subinterval of the
	cover-dominant range.
	
	\begin{theorem}\label{thm:stability}
		Fix $k\ge2$, $0<a\le1/(k+1)$, and $\eta>0$.  There exist $\eps>0$ and $n_0$ such
		that the following holds whenever $n\ge n_0$ and
$		an\le s\le\frac{n}{k+1}.$
		If $\cF\subseteq\binom{[n]}k$, $\nu(\cF)\le s$, and
		\[
		|\cF|\ge M_k(n,s)-\eps n^k,
		\]
		then some $S\in\binom{[n]}s$ satisfies
		\[
		|\cF\mathbin\triangle\cA(S)|\le\eta n^k.
		\]
	\end{theorem}
	
	The proof must overcome three separate losses, and each requires a different
	rigidity mechanism.
	
	\begin{enumerate}[label=\textup{(\arabic*)},leftmargin=8mm]
		\item \emph{From an integral matching bound to an approximate fractional
		cover.}  The natural linear-programming route is blocked by the fact that
		$\nu(\cF)\le s$ does not imply $\nu^*(\cF)\le s$.  We apply weak hypergraph
		regularity and pass to a bounded reduced $k$-graph.  Any sufficiently large
		fractional matching in the reduced graph can be rounded to an
		ordinary matching in $\cF$.  Duality in the reduced graph therefore yields a
		weight function $w:[n]\to[0,1]$ of average at most $x+o(1)$ that covers all
		but $o(n^k)$ edges of $\cF$.  This is the step that reconciles the ordinary
		matching problem with its fractional relaxation.
		
		\item \emph{From a sharp inequality to structural rigidity.}  Applied to the
		empirical law of $w$, a tail inequality can bound $|\cF|$, but the numerical
		bound alone gives no information about near extremizers.  We prove the
		equality-rigid statement
		\[
		\Prob(Z_1+\cdots+Z_k\ge1)\le1-(1-x)^k,
		\]
		for i.i.d. $Z_i\in[0,1]$ with $\E Z_i\le x\le1/(k+1)$, with equality only for
		$(1-x)\delta_0+x\delta_1$.  The numerical inequality also follows from the
		recent small-deviation theorem of Fu et al.~\cite{FuEtAl2026}; the new point
		needed here is the equality classification and its compactness consequence.
		Our equality analysis combines a reduction with a direct
		simplex cap estimate: equality in Gr\"unbaum's inequality
		\cite[Section~1]{TanganelliCastrillon2025} forces the relevant
		cap to point in a single coordinate direction, which ultimately singles out
		the Bernoulli law.  Compactness then forces $w$ to be almost $0$--$1$, proving
		\cref{thm:stability}.
		
		\item \emph{From stability to an exact extremal theorem.}  A routine removal argument cannot eliminate the $o(n^k)$ error from stability to get the exact statement. Instead, starting from
		an approximate center $S$, we move the few vertices with many missing link
		edges into the outside part.  A sparse matching theorem, including
		its equality case, determines the induced local family.  A greedy extension
		then shows that any unresolved excess would create a matching of size $s+1$.
		This equality-sensitive local step converts the stable cover into an exact one.
	\end{enumerate}
	
	The first two steps prove stability in the parameterized form stated as
	\cref{thm:parameterized-stability}; the third bootstraps it to the exact result.  The compactness and
	regularity arguments are qualitative, so the resulting value of $s_0(k)$ is not
	explicit.  \Cref{sec:probability} proves the probabilistic rigidity,
	\cref{sec:regularity} constructs the approximate fractional cover,
	\cref{sec:stability-proof} proves stability, and \cref{sec:exactification}
	completes the exact theorem.  The final section isolates the maximal-tail
	statement that would reach the asymptotically optimal barrier $\rho_k$, and
	records how Ling's recent theorem lowers the coefficient supplied by our
	self-contained argument from $k+1$ to $k+0.6$ for $k\ge5$.

    \section{Probabilistic rigidity}\label{sec:probability}
	
	We isolate the analytic statement needed later.  For $r\ge1$ and $T\ge r+1$, set
	\[
	q_{r,T}:=1-\left(1-\frac1T\right)^r.
	\]
	For $a\in\R$, let $\delta_a$ denote the unit point mass at $a$.
	For $r\ge1$, let
	\[
	\Delta_r:=\left\{d=(d_0,\ldots,d_r)\in[0,1]^{r+1}:
	\sum_{i=0}^r d_i=1\right\}.
	\]
	The notation $D=(D_0,\ldots,D_r)\sim\Dir(1,\ldots,1)$ means that $D$ is
	uniformly distributed on $\Delta_r$ with respect to normalized
	$r$-dimensional affine Lebesgue measure; $\Prob_D$ denotes probability under
	this law.
	
	The proof has three parts.  We first record a cap-volume lemma and the one
	probabilistic result from \cite{VlassisThomas2026}.  We then prove the required
	simplex estimate, including its equality case, directly from classical convex
	geometry.  Finally, we derive the sharp tail inequality, characterize all equality cases, and
	extract the compactness consequence used in the stability argument.
	
	\subsection{Geometric and probabilistic preliminaries}
	
	We first record the cap-volume consequence of the Brunn--Minkowski
	inequality needed for the direct simplex estimate.  Here a convex body in $\mathbb{R}^r$ means a compact convex set with
    nonempty interior, and $\vol$ denotes $r$-dimensional Lebesgue measure. If $C\subseteq\R^r$ is a
	convex body and $\xi\ne0$, put
	\[
	V_{C,\xi}(t):=
	\frac{\vol\bigl(C\cap\{z:\langle z,\xi\rangle\ge t\}\bigr)}{\vol(C)}.
	\]
	
	\begin{lemma}\label{lem:cap-concavity}
		The function $V_{C,\xi}^{\,1/r}$ is concave on the support of
		$V_{C,\xi}$.
	\end{lemma}
	
	\begin{proof}
		For $t_1,t_2$ in the interior of the support and $0\le\lambda\le1$, let
		\[
		C_i:=C\cap\{z:\langle z,\xi\rangle\ge t_i\},\qquad i=1,2.
		\]
		Convexity of $C$ gives
		\begin{align*}
			\lambda C_1+(1-\lambda)C_2
			\subseteq
			C\cap\{z:\langle z,\xi\rangle\ge\lambda t_1+(1-\lambda)t_2\}.
		\end{align*}
		The Brunn--Minkowski inequality \cite[Eq.~(2)]{Gardner2002}, in the form
		\[
		\vol\bigl(\lambda C_1+(1-\lambda)C_2\bigr)^{1/r}
		\ge \lambda\vol(C_1)^{1/r}+(1-\lambda)\vol(C_2)^{1/r},
		\]
		together with the preceding inclusion and division by $\vol(C)^{1/r}$,
		proves the asserted concavity.  The statement at the endpoints follows by
		continuity.
	\end{proof}
	
	We next record the chain-domination statement extracted from the proof of
	Vlassis and Thomas.  Let
	$(D_0,D_1,\dots,D_r)\sim\Dir(1,\dots,1)$ and define
	\[
	K_r(y):=\Prob_D\!\left(\sum_{i=1}^r y_iD_i\le1\right),
	\qquad y\in[0,\infty)^r.
	\]
	Suppose that independent mean-one variables $Y_i$ have nondegenerate two-point
	laws $Y_i\in\{\ell_i,h_i\}$, where $0\le\ell_i<1<h_i$.  Encode an outcome by its
	high set $A\subseteq[r]$, and write $y(A)_i=h_i$ for $i\in A$ and
	$y(A)_i=\ell_i$ otherwise.  Denote the product law of the high set by $\pi$.
	Thus, writing $p_i=(1-\ell_i)/(h_i-\ell_i)$,
	\[
	\pi(A)=\prod_{i\in A}p_i\prod_{i\notin A}(1-p_i).
	\]
	We call $H:2^{[r]}\to\R$ \emph{increasing} if $H(A)\le H(B)$ whenever
	$A\subseteq B$.
	
	\begin{theorem}[Vlassis--Thomas \cite{VlassisThomas2026}]\label{thm:chain}
		For every increasing function $H:2^{[r]}\to\R$, there is a maximal chain
		\[
		\mathcal C:\quad
		\varnothing=A_0\subset A_1\subset\cdots\subset A_r=[r]
		\]
		such that
		\[
		\E_\pi H\le \E_{\nu_{\mathcal C}}H.
		\]
		Here the chain measure is defined by
		\begin{align*}
			\nu_{\mathcal C}(A_j)
			&=K_r(y(A_j))-K_r(y(A_{j+1})) &&(0\le j<r),\\
			\nu_{\mathcal C}(A_r)&=K_r(y(A_r)).
		\end{align*}
		Moreover, every mean-one law on $[0,\infty)$ is a probability mixture of
		$\delta_1$ and mean-one two-point laws supported on
		$\{\ell,h\}$ with $0\le\ell<1<h$.
	\end{theorem}
	
	The first assertion is obtained by applying Lemma~6 of
	\cite{VlassisThomas2026} at every stage of the induction used to prove its
	Proposition~3; that lemma applies to an arbitrary increasing payoff.  The masses are
	nonnegative and telescope to $K_r(y(\varnothing))=1$.  The second assertion is
	Lemma~7 of that paper.
	
	\begin{remark}\label{rem:external-results}
		The numerical inequality in \cref{thm:tail-rigidity} below is the
		$\delta\ge1$ specialization of \cite[Theorem~1.1]{FuEtAl2026}, with
		$T=r+\delta$, after taking complements.  Indeed, if $S=\sum_iY_i$ and
		$\E S\le r$, then $\{S<\E S+\delta\}\subseteq\{S<T\}$, while
		\cite[Theorem~1.1]{FuEtAl2026} gives
		$\Prob(S<\E S+\delta)\ge(1-1/T)^r$.  That result therefore supplies the
		numerical inequality, but it does not characterize its equality cases, which are
		needed for our stability argument.  We prove both the inequality and the equality
		characterization in this regime without using the proof in \cite{FuEtAl2026}.
		The result from a 2026 preprint used in our proof is the chain-domination and
		mixture statement in \cref{thm:chain}; the source mapping is given immediately
		after that theorem.  The geometric estimate is derived from
		\cref{lem:cap-concavity} and Gr\"unbaum's inequality
		\cite[Theorem~2]{Grunbaum1960}, and every equality step used
		below is included in the proof.
	\end{remark}
	
	\subsection{A direct simplex rigidity estimate}
	
	We next derive the precise simplex estimate needed below from
	\cref{lem:cap-concavity} and Gr\"unbaum's inequality
	\cite[Theorem~2]{Grunbaum1960}.  We use
	the following equality statement \cite[Section~1]{TanganelliCastrillon2025}.  A
	cone means a set $\operatorname{conv}(B\cup\{a\})$, where
	$\operatorname{conv}$ denotes convex hull and $B$ is a compact
	$(r-1)$-dimensional convex set contained in a hyperplane and $a$ lies outside that
	hyperplane; $B$ and $a$ are called its base and apex.  If a halfspace through the
	centroid of a convex body attains the lower bound $(r/(r+1))^r$, then the body is
	such a cone, the cutting hyperplane is parallel to its base, and the halfspace of
	measure $(r/(r+1))^r$ contains the apex.
	
	\begin{lemma}\label{lem:simplex-rigidity}
		Let $r\ge2$, $T\ge r+1$, and $y\in[0,\infty)^r$ with
		$\sum_i y_i\ge T$.  Then
		\[
		K_r(y)\le q_{r,T}.
		\]
		Equality holds if and only if there is an index $j\in[r]$ such that
		$y_j=T$ and $y_i=0$ for every $i\ne j$.
	\end{lemma}
	
	\begin{proof}
		Let $g=(1/(r+1),\ldots,1/(r+1))$ be the centroid of $\Delta_r$.
		Write $L(d)=\sum_{i=1}^r y_id_i$ and $S=\sum_{i=1}^r y_i$.  Since $S\ge T>0$,
		$L$ is nonconstant on $\Delta_r$, and each of its level sets has zero
		$r$-dimensional simplex volume.  As $S/T\ge1$,
		\begin{equation}\label{eq:simplex-inclusion}
			1-K_r(y)=\Prob_D(L(D)\ge1)
			\ge \Prob_D(L(D)\ge S/T).
		\end{equation}
		Consider the centered simplex $C=\Delta_r-g$ in its $r$-dimensional linear span,
		which we identify isometrically with $\R^r$.
		Let $\xi$ be the orthogonal projection onto that span of
		$(0,y_1,\ldots,y_r)$.  This projection is nonzero because $L$ is nonconstant.
		If $z=d-g$, then
		\[
		\langle z,\xi\rangle=L(d)-\frac{S}{r+1},
		\qquad
		\max_{z\in C}\langle z,-\xi\rangle=\frac{S}{r+1};
		\]
		the displayed maximum is attained at the vertex $(1,0,\ldots,0)$.  Consequently, the
		last event in \eqref{eq:simplex-inclusion} is
		\[
		\left\{z\in C:\langle z,\xi\rangle\ge
		\alpha\frac{S}{r+1}\right\},
		\qquad
		\alpha:=\frac{r+1-T}{T}\in(-1,0].
		\]
		Indeed,
		\[
		\frac{S}{T}-\frac{S}{r+1}
		=\frac{r+1-T}{T}\frac{S}{r+1}.
		\]
		Let $V(t)=V_{C,\xi}(t)$ and put $h=S/(r+1)$.  Then $V(-h)=1$.
		Gr\"unbaum's centroid inequality \cite[Theorem~2]{Grunbaum1960} states
		that every halfspace whose boundary passes through the centroid of an
		$r$-dimensional convex body
		contains at least a fraction $(r/(r+1))^r$ of its volume.  Hence
		$V(0)\ge(r/(r+1))^r$, while \cref{lem:cap-concavity} and
		$\alpha h=(-\alpha)(-h)+(1+\alpha)0$ give
		\begin{align*}
			V(\alpha h)^{1/r}
			&\ge -\alpha V(-h)^{1/r}+(1+\alpha)V(0)^{1/r}\\
			&\ge -\alpha+(1+\alpha)\frac r{r+1}
			=1-\frac1T.
		\end{align*}
		Thus $\Prob_D(L(D)\ge S/T)=V(\alpha h)\ge(1-1/T)^r$.
		This proves the inequality.
		
		Suppose equality holds.  Then equality holds both in
		\eqref{eq:simplex-inclusion} and in Gr\"unbaum's inequality, since
		$1+\alpha>0$.  By the equality statement recalled above from
		\cite[Section~1]{TanganelliCastrillon2025}, the simplex is a cone whose base is
		parallel to a level set of $L$ and whose apex lies in the halfspace where $L$
		is larger.  Because the
		simplex is the convex hull of this apex and a set in a parallel supporting
		hyperplane, the apex is one vertex and all the remaining vertices lie in that
		hyperplane.  Thus the base is a facet and the apex is its unique opposite vertex.
		The values of $L$ at
		the vertices of $\Delta_r$ are
		\[
		0,y_1,\dots,y_r.
		\]
		The zeroth vertex cannot be the upper apex: its value is $0$, all other vertex
		values are nonnegative, and at least one is positive.  If the apex is the $j$th
		vertex, $L$ is constant on the opposite facet.  That facet contains the zeroth
		vertex and every vertex except the $j$th, so $y_i=0$ for every $i\ne j$.
		Thus $y_j=S\ge T$.  In this case
		\[
		K_r(y)=\Prob(D_j\le1/S)=1-\left(1-\frac1S\right)^r
		\le q_{r,T},
		\]
		with equality only when $S=T$.  Hence $y_j=T$.
		
		Conversely, if $y_j=T$ and $y_i=0$ for $i\ne j$, then
		\[
		K_r(y)=\Prob(D_j\le1/T)=1-\left(1-\frac1T\right)^r=q_{r,T},
		\]
		where the middle identity is the one-dimensional marginal formula for the uniform
		distribution on $\Delta_r$.
	\end{proof}
	
	For $r=1$, the tail inequality and equality characterization in the next theorem
	are exactly Markov's inequality and its equality conditions.
	
	\subsection{Tail rigidity and compactness}
	
	\begin{theorem}\label{thm:tail-rigidity}
		Let $r\ge1$ and $T\ge r+1$.  If $Y_1,\dots,Y_r$ are independent random variables
		such that $Y_i\ge0$ and $\E Y_i\le1$, then
		\[
		\Prob\!\left(\sum_{i=1}^rY_i\ge T\right)\le q_{r,T}.
		\]
		Equality holds if and only if, independently for every $i$,
		\[
		\Prob(Y_i=T)=\frac1T,\qquad \Prob(Y_i=0)=1-\frac1T.
		\]
	\end{theorem}
	
	\begin{proof}
		For $r=1$, Markov's inequality gives
		\[
		\Prob(Y_1\ge T)\le\frac{\E Y_1}{T}\le\frac1T=q_{1,T}.
		\]
		Equality in both steps forces $\E Y_1=1$, $Y_1=T$ on
		$\{Y_1\ge T\}$, and $Y_1=0$ otherwise.  This gives the asserted law for
		$r=1$.  We henceforth assume $r\ge2$ and proceed by induction on $r$.
		
		We first prove the theorem under the additional assumption that every variable
		has mean one.  Suppose initially that all $Y_i$ have nondegenerate two-point laws
		$\{\ell_i,h_i\}$.  Apply \cref{thm:chain} to the increasing payoff
		\[
		H(A)=\1\!\left\{\sum_i y(A)_i\ge T\right\}.
		\]
		Along the resulting chain, let $A_t$ be the first state with $H(A_t)=1$.  If no
		such state exists, the tail probability is zero.
		Otherwise, the states for which $H=1$ form a terminal segment of the chain, so the
		definition of the chain measure gives
		\[
		\Prob\!\left(\sum_iY_i\ge T\right)
		=\E_\pi H
		\le\E_{\nu_{\mathcal C}}H
		=K_r(y(A_t))
		\le q_{r,T},
		\]
		where the penultimate equality is a telescoping sum and the final inequality is
		\cref{lem:simplex-rigidity}.  This proves the required inequality for
		nondegenerate mean-one two-point systems.  If equality holds in the target bound,
		every displayed inequality is an equality.  After applying a permutation of the
		coordinate set $[r]$, \cref{lem:simplex-rigidity} allows us to assume
		\[
		A_t=\{1\},\qquad h_1=T,\qquad \ell_i=0\quad(2\le i\le r).
		\]
		Indeed, the vector $y(A_t)$ has exactly one positive coordinate, whereas each
		high coordinate is strictly larger than one; hence $A_t$ consists of that unique
		coordinate.
		
		We claim that also $\ell_1=0$.  Put $\ell=\ell_1$ and
		$u=T-\ell$.  Since $\E Y_1=1$,
		\[
		\Prob(Y_1=T)=\frac{1-\ell}{T-\ell}.
		\]
		Let $R=Y_2+\cdots+Y_r$.  Conditioning on $Y_1$ gives
		\[
		\Prob\!\left(\sum_iY_i\ge T\right)
		=\frac{1-\ell}{T-\ell}
		+\frac{T-1}{T-\ell}\Prob(R\ge T-\ell).
		\]
		Since $T-\ell>r$, the induction hypothesis in dimension $r-1$, applied at
		threshold $T-\ell$, yields
		\begin{align}
			\Prob\!\left(\sum_iY_i\ge T\right)
			&\le 1-\frac{T-1}{T-\ell}
			\left(1-\frac1{T-\ell}\right)^{r-1} \notag\\
			&=1-(T-1)\frac{(u-1)^{r-1}}{u^r}.\label{eq:ell-bound}
		\end{align}
		For $g(u)=(T-1)(u-1)^{r-1}/u^r$,
		\[
		\frac{g'(u)}{g(u)}=\frac{r-u}{u(u-1)}<0
		\qquad(T-1<u\le T),
		\]
		because $T\ge r+1$.  Thus the right side of \eqref{eq:ell-bound} is strictly
		smaller than $q_{r,T}$ when $\ell>0$.  Consequently $\ell_1=0$.
		
		Now condition on $Y_1=0$.  Since $Y_1$ equals $T$ with probability $1/T$, equality
		in the $r$-variable bound forces
		\[
		\Prob\!\left(\sum_{i=2}^rY_i\ge T\right)
		=1-\left(1-\frac1T\right)^{r-1}.
		\]
		The induction hypothesis shows that every remaining $Y_i$ has the asserted
		$\{0,T\}$ law.  This characterizes equality for nondegenerate two-point systems.
		
		We next allow deterministic mean-one coordinates.  If there are $d\ge1$ such
		coordinates and $m=r-d$ nondegenerate coordinates, then
		$m=0$ gives tail probability zero.  For $m\ge1$, apply the induction hypothesis to
		the $m$ nondegenerate coordinates at threshold $T-d=m+c$, where $c=T-r\ge1$.
		Their complementary probability is at least
		\[
		a_m(c):=\left(1-\frac1{m+c}\right)^m.
		\]
		The function $m\mapsto a_m(c)$ is strictly decreasing.  Indeed, for real $m>0$
		and $z=1/(m+c)$,
		\[
		\frac{\mathrm d}{\mathrm dm}\log a_m(c)
		=\log(1-z)+\frac{m}{(m+c-1)(m+c)}
		\le\log(1-z)+z<0.
		\]
		Since $T=r+c$, we have $a_r(c)=(1-1/T)^r$.  Thus
		$a_m(c)>a_r(c)$ when $m<r$, and the tail probability is strictly less than
		$q_{r,T}$.  This both proves the inequality and excludes deterministic
		coordinates in the case of equality.
		
		Now consider arbitrary mean-one laws on $[0,\infty)$.  By the mixture assertion in
		\cref{thm:chain}, each marginal can be generated by first choosing, independently
		across coordinates, either $\delta_1$ or a mean-one two-point component, and then
		sampling from that component.  Conditional on the latent components, the
		preceding argument bounds the tail probability by $q_{r,T}$.  If the unconditional
		probability equals $q_{r,T}$, then equality must hold for almost every latent
		tuple.  The preceding equality characterization, including the exclusion of deterministic
		coordinates, forces every chosen component in almost every tuple to be the unique
		mean-one law on $\{0,T\}$.  By Fubini's theorem, every marginal mixing measure is
		therefore concentrated on that single
		component.  Hence each original marginal is the $\{0,T\}$ law.  This
		proves the theorem when all means are one.
		
		Finally, put $m_i=\E Y_i$ and $\sigma=\sum_i m_i$ for general $m_i\le1$, and set
		\[
		\widehat Y_i:=Y_i+1-m_i,
		\qquad \widehat T:=T+r-\sigma.
		\]
		The variables $\widehat Y_i$ are independent, nonnegative, and have mean one,
		and
		\[
		\sum_iY_i\ge T
		\quad\Longleftrightarrow\quad
		\sum_i\widehat Y_i\ge\widehat T.
		\]
		Since $\widehat T\ge T\ge r+1$, the mean-one case gives
		\[
		\Prob\!\left(\sum_iY_i\ge T\right)
		\le q_{r,\widehat T}\le q_{r,T}.
		\]
		If $\sigma<r$, then $\widehat T>T$ and the last inequality is strict.  Hence
		equality forces $m_i=1$ for every $i$, and the mean-one equality characterization completes
		the proof.
	\end{proof}
	
	Scaling gives the form used for hypergraphs.
	
	\begin{corollary}\label{cor:threshold-rigidity}
		Let $0<x\le1/(k+1)$, and let $Z_1,\dots,Z_k$ be independent, identically
		distributed random variables in $[0,1]$ with $\E Z_i\le x$.  Then
		\[
		\Prob(Z_1+\cdots+Z_k\ge1)\le q_k(x)=1-(1-x)^k.
		\]
		Equality holds if and only if their common law is
		\[
		\mu_x:=(1-x)\delta_0+x\delta_1.
		\]
	\end{corollary}
	
	\begin{proof}
		Apply \cref{thm:tail-rigidity} to $Y_i=Z_i/x$ and $T=1/x$.
	\end{proof}

	Thus \cref{cor:threshold-rigidity} establishes exactly the tail bound and
	Bernoulli equality characterization required in
	\cref{thm:parameterized-stability} when
	$f_k=k+1$.
	
	We record the compactness consequence that turns convergence of the tail
	probabilities to the upper bound into convergence of the underlying laws.  For a
	probability measure $\mu$ on
	$[0,1]$, write $\mu^{\otimes k}$ for its $k$-fold product measure.
	
	\begin{lemma}\label{lem:compactness-rigidity}
		Let $f_k>k$.  Suppose that, for every $0<y\le1/f_k$, independent,
		identically distributed random variables $Y_1,\ldots,Y_k\in[0,1]$ with
		$\E Y_i\le y$ satisfy
		\[
		\Prob(Y_1+\cdots+Y_k\ge1)\le q_k(y),
		\]
		with equality if and only if their common law is
		$\mu_y=(1-y)\delta_0+y\delta_1$.  Suppose
		$x_j\to x\in(0,1/f_k]$ and $\mu_j$ are probability measures on $[0,1]$
		such that
		\[
		\int z\,\dd\mu_j(z)\le x_j+o(1)
		\]
		and
		\[
		\Prob_{\mu_j^{\otimes k}}(Z_1+\cdots+Z_k\ge1)
		\longrightarrow q_k(x).
		\]
		Then $\mu_j$ converges weakly to $\mu_x$.
	\end{lemma}
	
	\begin{proof}
		The space of probability measures on the compact interval $[0,1]$ is weakly
		compact.  Start with an arbitrary subsequence of $(\mu_j)$ and pass to a further
		subsequence, still indexed by $j$, such that $\mu_j\Rightarrow\mu$ for some
		probability measure $\mu$ on $[0,1]$.  Product measures then satisfy
		$\mu_j^{\otimes k}\Rightarrow\mu^{\otimes k}$.  Moreover, since
		$z\mapsto z$ is continuous on $[0,1]$,
		\[
		\int z\,\dd\mu(z)=\lim_j\int z\,\dd\mu_j(z)\le x.
		\]
		The set
		\[
		\mathcal E:=\{(z_1,\ldots,z_k)\in[0,1]^k:z_1+\cdots+z_k\ge1\}
		\]
		is closed.  The closed-set part of the Portmanteau theorem
		\cite[Theorem~2.1(iii)]{Billingsley1999} states that if
		$\nu_j\Rightarrow\nu$, then
		$\limsup_j\nu_j(F)\le\nu(F)$ for every closed set $F$.  Applying this with
		$\nu_j=\mu_j^{\otimes k}$, $\nu=\mu^{\otimes k}$ and $F=\mathcal E$, and then
		using the assumed convergence of the tail probabilities and the assumed
		tail bound gives
		\[
		q_k(x)\le\mu^{\otimes k}(\mathcal E)\le q_k(x).
		\]
		Thus equality holds in the assumed tail bound, and its equality
		characterization gives
		$\mu=\mu_x$.  We have shown that every subsequence of $(\mu_j)$ has a
		further subsequence converging weakly to $\mu_x$.  Therefore the full sequence
		converges weakly to $\mu_x$.
	\end{proof}
	
	\section{Approximate fractional covers}\label{sec:regularity}
	
	We next convert an ordinary matching constraint into a weight function that covers
	all but at most $\rho n^k$ edges.  More precisely, for fixed
	$\xi<1/k$ and $\rho>0$, the statement is uniform for
	$\nu(\cH)/n\le\xi$.
	
	We first state the precise regularity lemma.  If $V_1,\ldots,V_k$ are disjoint
	sets of the same size, define the corresponding crossing cell by
	\[
	\mathcal K(V_1,\ldots,V_k)
	:=\{E\in\tbinom{V_1\cup\cdots\cup V_k}{k}:|E\cap V_i|=1
	\text{ for every }i\}.
	\]
	For a $k$-graph $\cH$, the density of this cell is
	\[
	d_{\cH}(V_1,\ldots,V_k)
	:=\frac{|\cH\cap\mathcal K(V_1,\ldots,V_k)|}
	{|V_1|\cdots|V_k|}.
	\]
	The tuple $(V_1,\ldots,V_k)$ is \emph{$\eps$-regular} if, for every
	$U_i\subseteq V_i$ with $|U_i|\ge\eps|V_i|$,
	\[
	\bigl|d_{\cH}(U_1,\ldots,U_k)-d_{\cH}(V_1,\ldots,V_k)\bigr|
	\le\eps.
	\]
	The following vertex-partition form of weak hypergraph regularity is the only
	regularity result used in the paper; see \cite{Chung1991,FranklRodl1992}.  We state
	the complete form needed for the reduced-hypergraph construction.
	
	\begin{lemma}\label{lem:weak-regularity}
		For every $k\ge2$, $t_0\ge k$ and $\eps>0$, there are integers
		$T_0=T_0(k,t_0,\eps)$ and $n_0=n_0(k,t_0,\eps)$
		such that every $k$-graph on $n\ge n_0$ vertices admits a partition
		\[
		V=V_0\cup V_1\cup\cdots\cup V_t
		\]
		for which $t_0\le t\le T_0$, $|V_0|\le\eps n$, the sets
		$V_1,\ldots,V_t$ have the same size, and all but at most
		$\eps\binom tk$ choices $1\le i_1<\cdots<i_k\le t$ give an
		$\eps$-regular tuple $(V_{i_1},\ldots,V_{i_k})$.
	\end{lemma}
	
	The feature used below is that if an $\eps$-regular cell has density at
	least $d>\eps$, then every choice of subsets occupying at least an $\eps$-fraction
	of each participating cluster still contains an edge.  This permits a fractional
	matching of the reduced $k$-graph to be rounded greedily inside the original
	$k$-graph.
	
	For a finite $k$-graph $R$, a \emph{fractional matching} is a function
	$\varphi:E(R)\to[0,\infty)$ such that
	\[
	\sum_{e\ni v}\varphi(e)\le1
	\quad\text{for every }v\in V(R).
	\]
	Its size is $\sum_{e\in E(R)}\varphi(e)$, and $\nu^*(R)$ denotes the maximum
	size.  A \emph{fractional vertex cover} is a function
	$z:V(R)\to[0,1]$ such that $\sum_{v\in e}z(v)\ge1$ for every $e\in E(R)$.
	Finite-dimensional linear-programming duality says that the minimum weight
	$\sum_vz(v)$ of such a cover equals $\nu^*(R)$.  The usual dual only requires
	$z(v)\ge0$; truncating every coordinate at $1$ preserves all cover constraints,
	so the displayed range $[0,1]$ entails no loss.
	
	Observe that \cref{cor:threshold-rigidity} already gives the required
	asymptotic cover bound under a fractional matching constraint.  Indeed, if
	$\cH\subseteq\binom{[n]}k$, $0<s\le n/(k+1)$, and $\nu^*(\cH)\le s$, choose a
	fractional vertex cover $w:[n]\to[0,1]$ of total weight at most $s$.  For
	independent uniform vertices $V_1,\ldots,V_k\in[n]$, set $Z_i=w(V_i)$.  Then
	$\E Z_i\le s/n$, and every edge of $\cH$ contributes its $k!$ orderings to the
	event $Z_1+\cdots+Z_k\ge1$.  Consequently,
	\[
	k!|\cH|
	\le n^k\Prob(Z_1+\cdots+Z_k\ge1)
	\le n^k-(n-s)^k.
	\]
	For integral $s$, the last quantity divided by $k!$ equals
	$M_k(n,s)+O_k(n^{k-1})$.  Thus the remaining issue is precisely that
	$\nu(\cH)\le s$ need not imply $\nu^*(\cH)\le s$; the next lemma supplies the
	quantified substitute required for the ordinary matching problem.
	
	We now prove the weight-function statement used in the stability argument.
	
	\begin{lemma}\label{lem:approx-cover}
		Fix $k\ge2$, $0\le\xi<1/k$, and $\rho>0$.  There exists $n_0$ such that, for every
		$n\ge n_0$, the following holds.  If $0\le x\le\xi$ and
		$\cH\subseteq\binom{[n]}k$ satisfies $\nu(\cH)\le xn$, then there is a function
		$w:[n]\to[0,1]$ such that
		\[
		\sum_{v=1}^n w(v)\le(x+\rho)n
		\]
		and all but at most $\rho n^k$ edges $E\in\cH$ satisfy
		\[
		\sum_{v\in E}w(v)\ge1.
		\]
	\end{lemma}
	
	\begin{proof}
		Put $C_k=1/(2(k-2)!)$.
		Choose $a$ with $0<a<\rho/4$.  Since $\xi+a>\xi$, choose $\gamma>0$ such that
		\[
		(1-\gamma)(\xi+a)>\xi.
		\]
		Next choose $\eps>0$ such that
		\[
		\eps<\min\{\gamma,\rho/8\}
		\quad\text{and}\quad
		(1-\gamma)(1-\eps)(\xi+a)>\xi,
		\]
		and then choose $d$ with $\eps<d<\rho/8$.  The function
		$(1-\gamma)(1-\eps)(x+a)-x$ is decreasing in $x$, so these choices ensure,
		uniformly for $0\le x\le\xi$, that
		\begin{equation}\label{eq:rounding-margin}
			(1-\gamma)(1-\eps)(x+a)>x.
		\end{equation}
		Finally choose $t_0\ge k$ sufficiently large that
		\[
		\frac{C_k}{t_0}+\eps+d<\rho,
		\]
		where $C_k$ is the constant in the repeated-cluster estimate below.
		
		Apply \cref{lem:weak-regularity} to $\cH$, and let
		$m=|V_1|=\cdots=|V_t|$.  Form a reduced $k$-graph $R$ on $[t]$: a set
		$\{i_1,\ldots,i_k\}\in\binom{[t]}k$ is an edge of $R$ if
		$(V_{i_1},\ldots,V_{i_k})$ is $\eps$-regular and has density at least $d$.
		
		We claim that the fractional matching number of $R$ is at most $(x+a)t$.  If not,
		let $(\varphi_e)_{e\in E(R)}$ be a fractional matching of larger total weight.
		For each $e\in E(R)$, plan to choose
		\[
		m_e=\lfloor(1-\gamma)\varphi_e m\rfloor
		\]
		pairwise disjoint edges from the corresponding crossing cell.  For every cluster
		$V_i$, its planned load is at most
		\[
		\sum_{e\ni i}m_e
		\le(1-\gamma)m\sum_{e\ni i}\varphi_e
		\le(1-\gamma)m.
		\]
		Greedily choose the edges in any order.  Whenever a cell is used, every
		participating cluster still has at least $\gamma m\ge\eps m$ unused vertices.
		By regularity, the density in the remaining subsets is at least
		$d-\eps>0$, so another edge is available.  Hence the plan can be completed and
		gives a
		matching in $\cH$ of size
		\begin{align*}
			\sum_{e\in E(R)}m_e
			\ge(1-\gamma)m\nu^*(R)-\binom tk
			>(1-\gamma)(1-\eps)(x+a)n-\binom{T_0}k>xn
		\end{align*}
		for sufficiently large $n$, by \eqref{eq:rounding-margin}.  This contradicts the
		hypothesis and proves the claim.
		
		By LP duality, $R$ has a fractional vertex cover $(z_i)_{i=1}^t$ with
		$z_i\in[0,1]$ and
		\[
		\sum_i z_i\le(x+a)t.
		\]
		Lift it by setting $w(v)=z_i$ for $v\in V_i$ and $w(v)=1$ for $v\in V_0$.
		Since $tm=n-|V_0|$ and $a+\eps<\rho$,
		\[
		\sum_vw(v)
		\le m(x+a)t+|V_0|
		\le(x+a+\eps)n
		<(x+\rho)n.
		\]
		
		Every edge lying in a cell indexed by an edge of $R$ is covered.  An uncovered edge
		must therefore use two vertices in one of $V_1,\ldots,V_t$, lie in an irregular crossing
		cell, or lie in a crossing cell of density below $d$.  The respective numbers are
		at most
		\[
		\frac{C_k}{t_0}n^k,\qquad \eps n^k,\qquad dn^k.
		\]
		For the first estimate, sum
		$\binom m2\binom{n}{k-2}$ over $V_1,\ldots,V_t$ and use
		$m\le n/t$.  For each of the other two estimates, sum over at most
		$\binom tk$ crossing cells, each containing at most $m^k$ possible edges.
		Edges meeting $V_0$ are covered because their weight sum is at least one.  The
		choice of $t_0$ makes the total smaller than $\rho n^k$.
	\end{proof}
	
	\section{Proof of stability}\label{sec:stability-proof}
	
	The stability argument admits the following parameterized form.

	\begin{theorem}\label{thm:parameterized-stability}
		Fix $k\ge2$ and $f_k>k$.  Suppose that, for every $0<x\le1/f_k$,
		independent, identically distributed random variables
		$Z_1,\ldots,Z_k\in[0,1]$ with $\E Z_i\le x$ satisfy
		\[
		\Prob(Z_1+\cdots+Z_k\ge1)\le q_k(x),
		\]
		with equality if and only if their common law is
		$\mu_x=(1-x)\delta_0+x\delta_1$.  Let $0<a\le1/f_k$ and $\eta>0$.
		There exist $\eps>0$ and $n_0$ such that the following holds whenever
		$n\ge n_0$ and
$		an\le s\le\frac{n}{f_k}.$
		If $\cF\subseteq\binom{[n]}k$, $\nu(\cF)\le s$, and
		\[
		|\cF|\ge M_k(n,s)-\eps n^k,
		\]
		then some $S\in\binom{[n]}s$ satisfies
		\[
		|\cF\mathbin\triangle\cA(S)|\le\eta n^k.
		\]
	\end{theorem}

	\begin{proof}[Proof of \cref{thm:parameterized-stability}]
		Suppose the result is false.  Choosing the error in the size hypothesis to be
		$1/j$ and then taking $n_j\ge j$, we obtain sequences $n_j\to\infty$, integers
		$s_j$ with
		\[
		a\le x_j:=\frac{s_j}{n_j}\le\frac1{f_k},
		\]
		and families $\cF_j\subseteq\binom{[n_j]}k$ such that
		\begin{align}
			\nu(\cF_j)&\le s_j,\label{eq:stab-match}\\
			|\cF_j|&\ge M_k(n_j,s_j)-o(n_j^k),\label{eq:stab-size}
		\end{align}
		but $|\cF_j\mathbin\triangle\cA(S)|>\eta n_j^k$ for every
		$S\in\binom{[n_j]}{s_j}$.
		Passing to a subsequence, assume $x_j\to x\in[a,1/f_k]$.
		
		For each positive integer $m$, let $N(m)$ be the threshold supplied by
		\cref{lem:approx-cover} with $\xi=1/f_k$ and $\rho=1/m$.  Since
		$n_j\to\infty$, there are integers
		$m_j\to\infty$ such that $n_j\ge N(m_j)$ after passing to a subsequence.  Apply
		the lemma with $\rho_j=1/m_j$.  We obtain weights $w_j:[n_j]\to[0,1]$ with
		\[
		\frac1{n_j}\sum_vw_j(v)\le x_j+o(1)
		\]
		such that, for
		\[
		\cG_j:=\left\{E\in\binom{[n_j]}k:\sum_{v\in E}w_j(v)\ge1\right\},
		\]
		we have
		\begin{equation}\label{eq:FminusG}
			|\cF_j\setminus\cG_j|=o(n_j^k).
		\end{equation}
		
		Let
		$
		\mu_j:={n_j}^{-1}\sum_{v\in[n_j]}\delta_{w_j(v)}
		$
		be the empirical probability measure of the weights.  Sampling
		$k$ vertices with replacement instead of without replacement changes any event by
		at most $\binom{k}{2}/n_j$.  Therefore
		\begin{equation}\label{eq:G-probability}
			\frac{|\cG_j|}{\binom{n_j}k}
			=\Prob_{\mu_j^{\otimes k}}(Z_1+\cdots+Z_k\ge1)+o(1).
		\end{equation}
		Also, uniformly for $x_j$ in the present compact interval,
		\begin{equation}\label{eq:cover-density}
			\frac{M_k(n_j,s_j)}{\binom{n_j}k}
			=1-\frac{\binom{n_j-s_j}k}{\binom{n_j}k}
			=q_k(x_j)+O_k(n_j^{-1}).
		\end{equation}
		
		Equations \eqref{eq:stab-size} and \eqref{eq:FminusG} give a lower bound
		$q_k(x)-o(1)$ for \eqref{eq:G-probability}.  To obtain the reverse bound, choose a
		subsequence on which the limsup of the probabilities is attained.  Since the
		space of probability measures on $[0,1]$ is weakly compact, pass to a further
		subsequence for which $\mu_j$ converges weakly to a probability measure $\mu$.
		Then $\mu_j^{\otimes k}$ converges weakly to $\mu^{\otimes k}$, and the continuity
		of $z\mapsto z$ on $[0,1]$ gives $\int z\,\dd\mu(z)\le x$.  Because
		$\{z_1+\cdots+z_k\ge1\}$ is closed, the closed-set part of the
		Portmanteau theorem \cite[Theorem~2.1(iii)]{Billingsley1999}, namely
		$\limsup_j\nu_j(F)\le\nu(F)$ whenever $\nu_j\Rightarrow\nu$ and $F$ is
		closed, together with the assumed tail bound, gives
		\[
		\limsup_j\Prob_{\mu_j^{\otimes k}}(Z_1+\cdots+Z_k\ge1)
		\le\Prob_{\mu^{\otimes k}}(Z_1+\cdots+Z_k\ge1)
		\le q_k(x).
		\]
		Consequently $|\cG_j|=(q_k(x_j)+o(1))\binom{n_j}k$.  The lower bound on
		$|\cF_j|$, together with $|\cF_j\setminus\cG_j|=o(n_j^k)$, now implies first that
		$|\cF_j|=M_k(n_j,s_j)+o(n_j^k)$ and then that
		\begin{equation}\label{eq:F-G-close}
			|\cF_j\mathbin\triangle\cG_j|=o(n_j^k).
		\end{equation}
		
		Equation \eqref{eq:G-probability} now shows that the product tail probabilities
		converge to $q_k(x)$.  Hence \cref{lem:compactness-rigidity} gives
		$\mu_j\Rightarrow\mu_x$.  Fix
		$0<\gamma<1/k$ and put
		\begin{align*}
			A_j&:=\{v:w_j(v)>1-\gamma\},\\
			I_j&:=\{v:\gamma<w_j(v)\le1-\gamma\}.
		\end{align*}
		Weak convergence gives
		\[
		|A_j|=xn_j+o(n_j)=s_j+o(n_j),
		\qquad |I_j|=o(n_j).
		\]
		Every edge of $\cG_j$ that avoids $A_j$ must meet $I_j$, since otherwise its
		weight sum is at most $k\gamma<1$.  Hence
		\[
		|\cG_j\setminus\cA(A_j)|
		\le |I_j|\binom{n_j-1}{k-1}=o(n_j^k).
		\]
		Moreover, $|A_j|=s_j+o(n_j)$ and \eqref{eq:cover-density} show that
		\[
		\bigl||\cA(A_j)|-M_k(n_j,s_j)\bigr|
		\le \bigl||A_j|-s_j\bigr|\binom{n_j-1}{k-1}
		=o(n_j^k).
		\]
		Since $|\cG_j|=M_k(n_j,s_j)+o(n_j^k)$, it follows that
		\[
		|\cA(A_j)\setminus\cG_j|
		=|\cA(A_j)|-|\cG_j|+|\cG_j\setminus\cA(A_j)|=o(n_j^k),
		\]
		and hence $|\cG_j\mathbin\triangle\cA(A_j)|=o(n_j^k)$.  Choose an
		$s_j$-set $S_j$ with
		$|A_j\mathbin\triangle S_j|=\bigl||A_j|-s_j\bigr|=o(n_j)$.  The same
		one-vertex estimate gives
		$|\cA(A_j)\mathbin\triangle\cA(S_j)|=o(n_j^k)$.
		Together with \eqref{eq:F-G-close}, this contradicts the assumed inequality
		$|\cF_j\mathbin\triangle\cA(S_j)|>\eta n_j^k$.
	\end{proof}

	\begin{proof}[Proof of \cref{thm:stability}]
		Apply \cref{thm:parameterized-stability} with $f_k=k+1$.  Its
		probabilistic hypothesis follows from \cref{cor:threshold-rigidity}.
	\end{proof}
	
	\section{Proof of the exact bound}\label{sec:exactification}
	
	We first isolate the matching statement used when the ratio $b/N$ is at most a
	constant $\theta_k$ supplied by the lemma below.  Its equality statement will be
	used in the local argument.
	For a $k$-graph $\cJ$, let $\tau(\cJ)$ denote its vertex-cover number, the minimum
	size of a set meeting every edge of $\cJ$.
	
	\begin{lemma}\label{lem:sparse-emc}
		For every fixed $k\ge2$, there are constants $\theta_k>0$ and $N_k$ such that,
		whenever $N\ge N_k$ and $0\le b\le\theta_kN$, every
		$\cJ\subseteq\binom{[N]}k$ with $\nu(\cJ)\le b$ satisfies
		\[
		|\cJ|\le\binom Nk-\binom{N-b}k.
		\]
		Equality holds if and only if
		\[
		\cJ=\cA_{[N]}(C)
		\]
		for some $C\in\binom{[N]}b$.
	\end{lemma}
	
	\begin{proof}
		For $k=2$, the Erd\H{o}s--Gallai theorem gives
		\[
		|\cJ|\le
		\max\left\{\binom{2b+1}{2},
		\binom N2-\binom{N-b}{2}\right\}.
		\]
		If $b\le N/5$ and $N\ge5$, then for $b>0$ the second term is strictly larger
		than the first, since $N\ge5b>(5b+3)/2$; the case $b=0$ is trivial.  The
		equality statement in the
		Erd\H{o}s--Gallai theorem therefore shows that equality is possible only for
		the graph consisting of all pairs meeting a fixed $b$-set.  Thus the assertion
		holds for $k=2$.
		
		Suppose henceforth that $k\ge3$.
		Frankl and Kupavskii \cite{FranklKupavskii2022} proved the displayed bound for all
		sufficiently large $b$ (with $k$ fixed) provided
		\[
		N\ge\frac{5k-2}{3}\,b.
		\]
		Choose
		\[
		0<\theta_k<\min\left\{\frac3{5k-2},\frac1{2k^3}\right\}.
		\]
		There is a constant $b_0=b_0(k)$ above which the Frankl--Kupavskii result
		applies.  For each of the finitely many integers $0\le b<b_0$, Erd\H{o}s's
		classical theorem gives the same bound once $N$ is sufficiently large
		\cite{Erdos1965}.  Taking $N_k$ larger than all these thresholds proves the
		numerical assertion.
		
		It remains to characterize all equality cases.  If $b=0$, then
		$\nu(\cJ)=0$ forces $\cJ=\varnothing=\cA_{[N]}(\varnothing)$, so assume
		$b\ge1$ and $|\cJ|=\binom Nk-\binom{N-b}k$.  If
		$\nu(\cJ)\le b-1$, the numerical assertion just proved, applied with parameter
		$b-1$, gives
		\[
		|\cJ|\le\binom Nk-\binom{N-b+1}k
		<\binom Nk-\binom{N-b}k,
		\]
		a contradiction.  Hence $\nu(\cJ)=b$.
		
		Our choice of $\theta_k$ ensures $N>2k^3b$.  The theorem of Bollob\'as,
		Daykin and Erd\H{o}s \cite{BollobasDaykinErdos1976} states in this range that,
		if $\nu(\cJ)\le b$ and no $b$-set meets every edge of $\cJ$, equivalently if
		$\tau(\cJ)>b$, then
		\[
		|\cJ|\le
		\binom Nk-\binom{N-b}k+1-\binom{N-b-k}{k-1}
		<\binom Nk-\binom{N-b}k,
		\]
		where the last inequality holds after increasing $N_k$ if necessary.  Every
		vertex cover must use distinct vertices to meet the $b$ edges of a matching, so
		$\tau(\cJ)\ge\nu(\cJ)=b$.  The preceding strict bound therefore implies
		$\tau(\cJ)=b$.  If $C$ is a vertex cover of size $b$, then
		$\cJ\subseteq\cA_{[N]}(C)$; equality of their sizes forces
		$\cJ=\cA_{[N]}(C)$.  Conversely, every $\cA_{[N]}(C)$ with $|C|=b$ has
		matching number at most $b$ and the displayed number of edges.
	\end{proof}
	
	\begin{lemma}\label{lem:exactification}
		Fix $k\ge2$, a real number $f_k>k$, and $L\ge f_k$.  There are $\delta>0$
		and $s_1$ such that the following
		holds whenever $s\ge s_1$ and
		\[
		f_k s\le n\le Ls.
		\]
		If $\cF\subseteq\binom{[n]}k$, $\nu(\cF)\le s$, and some
		$S\in\binom{[n]}s$ satisfies
		\[
		|\cF\mathbin\triangle\cA(S)|\le\delta n^k,
		\]
		then $|\cF|\le M_k(n,s)$.  Moreover, equality holds if and only if
		$\cF=\cA(S')$ for some $S'\in\binom{[n]}s$.
	\end{lemma}
	
	\begin{proof}
		Put $\alpha=f_k-k>0$.  It is enough to consider the case
		$|\cF|\ge M_k(n,s)$.  Let
		$W=[n]\setminus S$, and let
		\begin{align*}
			h&:=|\cA(S)\setminus\cF|,\\
			g&:=|\cF\cap\tbinom Wk|.
		\end{align*}
		Since $|\cF|=M_k(n,s)-h+g$, we have
		\begin{equation}\label{eq:g-at-least-h}
			g\ge h.
		\end{equation}
		
		For $v\in S$, let
		\[
		m(v):=\left|\left\{A\in\binom W{k-1}:\{v\}\cup A\notin\cF\right\}\right|.
		\]
		Set
		\[
		D:=\binom{\lfloor \alpha s/2\rfloor}{k-1},
		\qquad B:=\{v\in S:m(v)\ge D\},
		\qquad b:=|B|.
		\]
		The missing edges counted for different vertices are distinct and all belong to
		$\cA(S)\setminus\cF$, so
		\begin{equation}\label{eq:bD}
			bD\le h\le|\cF\mathbin\triangle\cA(S)|\le\delta n^k.
		\end{equation}
		For all sufficiently large $s$,
		\[
		D=\binom{\lfloor \alpha s/2\rfloor}{k-1}
		\ge c_{k,f_k}s^{k-1},
		\qquad
		c_{k,f_k}:=\left(\frac{\alpha}{3(k-1)}\right)^{k-1}.
		\]
		Indeed, $\lfloor\alpha s/2\rfloor\ge\alpha s/3$ once $s$ is sufficiently
		large, and
		$\binom mr\ge(m/r)^r$ for $m\ge r$.  It follows from \eqref{eq:bD} and
		$n\le Ls$ that
		\[
		b\le \frac{L^k}{c_{k,f_k}}\,\delta s.
		\]
		Choose $\delta>0$ so small that
		\[
		\frac{L^k}{c_{k,f_k}}\delta
		\le\min\left\{\frac{\alpha}{3},(f_k-1)\theta_k\right\}.
		\]
		Since $n-s+b\ge n-s\ge(f_k-1)s$, we then have
		\begin{equation}\label{eq:b-small}
			b\le \alpha s/3
			\quad\text{and}\quad
			b\le\theta_k(n-s+b),
		\end{equation}
		where $\theta_k$ is from \cref{lem:sparse-emc}.  Increase $s_1$ if necessary so
		that $N=n-s+b\ge N_k$, $\alpha s\ge6$, and
		$\lfloor\alpha s/2\rfloor\ge k-1$.
		
		Let $U=W\cup B$ and $Q=S\setminus B$, so that
		$[n]=Q\mathbin{\dot\cup}U$ and $|U|=N$.  We first claim that
		\begin{equation}\label{eq:local-matching}
			\nu(\cF[U])\le b.
		\end{equation}
		Suppose otherwise, and choose a matching $\cM_0\subseteq\cF[U]$ of size
		$b+1$.  We greedily extend $\cM_0$ by one edge of the form
		$\{q\}\cup A_q$, with $A_q\in\binom W{k-1}$, for each $q\in Q$.  At any step,
		at most $|Q|-1=s-b-1$ such edges have previously been selected.  The matching
		$\cM_0$ uses at most $k(b+1)$ vertices of $W$, so the number of currently unused
		vertices of $W$ is at least
		\begin{align*}
			|W|-k(b+1)-(k-1)(s-b-1)
			=n-ks-b-1
			\ge\alpha s-b-1
			\ge\frac{2\alpha s}{3}-1
			\ge\lfloor\alpha s/2\rfloor.
		\end{align*}
		There are therefore at least $D$ candidate $(k-1)$-sets contained in the unused
		part of $W$.  Because $q\notin B$, fewer than $D$ members of
		$\binom W{k-1}$ fail to form an edge with $q$.  Hence at least one candidate gives an
		edge $\{q\}\cup A_q\in\cF$, and the greedy construction can continue.  At the
		end we obtain $(b+1)+(s-b)=s+1$ pairwise disjoint edges of $\cF$, a
		contradiction.  This proves \eqref{eq:local-matching}.
		
		Write
		\[
		h_B:=|\cA_U(B)\setminus\cF[U]|\le h.
		\]
		Since $N-b=n-s$, we have
		\begin{align*}
			|\cF[U]|
			&=\binom Nk-\binom{n-s}k-h_B+g\\
			&=M_k(N,b)-h_B+g
			\ge M_k(N,b),
		\end{align*}
		where the last inequality follows from
		$g\ge h\ge h_B$.  On the other hand, \eqref{eq:local-matching},
		\eqref{eq:b-small}, and \cref{lem:sparse-emc} give
		$|\cF[U]|\le M_k(N,b)$.  Consequently,
		\[
		g=h=h_B,
		\qquad |\cF[U]|=M_k(N,b).
		\]
		The equality statement in \cref{lem:sparse-emc} yields a set
		$B'\in\binom Ub$ such that
		\[
		\cF[U]=\cA_U(B').
		\]
		
		The equality $h=h_B$ shows that every missing member of $\cA(S)$ is contained
		in $U$ and meets $B$.  Hence every $k$-set meeting $Q$ belongs to $\cF$.
		Since $[n]=Q\mathbin{\dot\cup}U$, it follows that
		\[
		\cF
		=\cA_{[n]}(Q)\cup\cF[U]
		=\cA_{[n]}(Q)\cup\cA_U(B')
		=\cA_{[n]}(Q\cup B').
		\]
		Finally, $Q\cap B'=\varnothing$ and
		$|Q\cup B'|=(s-b)+b=s$.  Thus equality forces the asserted form.  Conversely,
		every family $\cA(S')$ with $|S'|=s$ has $M_k(n,s)$ edges and matching number at
		most $s$.
	\end{proof}
	
	The parameterized formulation makes the dependence of the Erd\H{o}s matching conjecture range on the
	probabilistic estimate explicit.
	
	\begin{corollary}\label{cor:parameterized-emc}
		Fix $k\ge2$ and let $f_k\ge\rho_k$.  Suppose that, for every
		$0<x\le1/f_k$, independent, identically distributed random variables
		$Z_1,\ldots,Z_k\in[0,1]$ with $\E Z_i\le x$ satisfy
		\[
		\Prob(Z_1+\cdots+Z_k\ge1)\le q_k(x),
		\]
		with equality if and only if their common law is
		$\mu_x=(1-x)\delta_0+x\delta_1$.  Then
		there exists $s_0=s_0(k,f_k)$ such that the following holds.  If
		$s\ge s_0$, $n\ge f_k s$, and
		$\cF\subseteq\binom{[n]}k$ satisfies $\nu(\cF)\le s$, then
		\[
		|\cF|\le M_k(n,s).
		\]
		Equality holds if and only if $\cF=\cA(S)$ for some
		$S\in\binom{[n]}s$.
	\end{corollary}
	
	\begin{proof}
		Let $\theta_k$ be supplied by \cref{lem:sparse-emc}, and fix
		\[
		L_k>\max\left\{f_k,\frac1{\theta_k}\right\}.
		\]
		Apply \cref{lem:exactification} with the parameters $f_k$ and $L=L_k$,
		obtaining $\delta>0$, and apply \cref{thm:parameterized-stability} with the same $f_k$,
		$a=1/L_k$, and $\eta=\delta$.
		Take $s$ sufficiently large, let $n\ge f_k s$, and suppose that
		$\cF\subseteq\binom{[n]}k$ satisfies $\nu(\cF)\le s$ and
		$|\cF|\ge M_k(n,s)$.
		
		If $n\le L_ks$, then $1/L_k\le s/n\le1/f_k$, so
		\cref{thm:parameterized-stability} gives an $S\in\binom{[n]}s$ such that
		$|\cF\mathbin\triangle\cA(S)|\le\delta n^k$.  The equality-sensitive form of
		\cref{lem:exactification} now shows that
		$|\cF|=M_k(n,s)$ and $\cF=\cA(S')$ for some
		$S'\in\binom{[n]}s$.
		
		If $n>L_ks$, then $s<\theta_kn$.  Applying \cref{lem:sparse-emc} directly
		with $(N,b)=(n,s)$ gives $|\cF|\le M_k(n,s)$, with equality only when
		$\cF=\cA(S')$ for an $s$-set $S'$.  Since we assumed
		$|\cF|\ge M_k(n,s)$, equality holds here as well.  Thus every family of size at
		least $M_k(n,s)$ is one of the asserted cover families.  Taking $s_0(k,f_k)$
		large enough for all the preceding applications proves the result.
	\end{proof}
	
	\begin{proof}[Proof of \cref{thm:main}]
		By \cref{cor:threshold-rigidity}, the tail bound and its Bernoulli equality
		characterization required in \cref{cor:parameterized-emc} hold for
		$f_k=k+1$.  Since
		$k+1>\rho_k$, \cref{cor:parameterized-emc} with $f_k=k+1$ gives the stated
		bound and its equality classification.  The asymptotic comparison defining
		$\rho_k$ shows that, after increasing $s_0(k)$ if necessary, the cover term is
		larger than the clique term at $n=(k+1)s$.  Since
		$M_k(n+1,s)-M_k(n,s)=\binom n{k-1}-\binom{n-s}{k-1}>0$, the same holds
		throughout $n\ge(k+1)s$, so this is precisely the prediction of
		\cref{conj:emc} in the asserted range.
	\end{proof}
	
	\section{Concluding remarks}\label{sec:optimal-conjecture}
	
	The preceding proof isolates a single analytic barrier.  Recall that $\rho_k$ is
	the unique number in $(k,k+1)$ satisfying
$	\rho_k^k-(\rho_k-1)^k=k^k.$
	The following is exactly the cover-dominant part of the i.i.d. maximal-tail
	conjecture of \L{}uczak, Mieczkowska and \v{S}ileikis
	\cite{LuczakMieczkowskaSileikis2017}.  We include the strict equality statement
	because it is the form required by the stability argument.
	
	\begin{conjecture}\label{conj:optimal-tail}
		Let $0<x\le1/\rho_k$, and let $Z_1,\dots,Z_k$ be independent, identically
		distributed random variables in $[0,1]$ with $\E Z_i\le x$.  Then
		\[
		\Prob(Z_1+\cdots+Z_k\ge1)\le q_k(x)=1-(1-x)^k.
		\]
		If $0<x<1/\rho_k$, equality holds if and only if their common law is
		$\mu_x=(1-x)\delta_0+x\delta_1$.
	\end{conjecture}
	
	The interval is best possible.  At $x=1/\rho_k$, the additional law
	\[
	\left(1-\frac{k}{\rho_k}\right)\delta_0
	+\frac{k}{\rho_k}\delta_{1/k}
	\]
	has mean $1/\rho_k$ and tail probability $(k/\rho_k)^k=q_k(1/\rho_k)$.  Thus uniqueness must fail at the endpoint.  For $x>1/\rho_k$, the analogous law with mass $kx$ at $1/k$ has tail probability $(kx)^k>q_k(x)$, while for $x>1/k$ the deterministic law $\delta_{1/k}$ already violates the proposed bound.
	
	More generally, let $f_k\ge\rho_k$.  Whenever the inequality in
	\cref{conj:optimal-tail}, together with Bernoulli uniqueness, is established
	throughout $0<x\le1/f_k$, \cref{cor:parameterized-emc} yields the exact
	cover-side range
	$n\ge f_k s.$
	Thus $f_k=k+1$, supplied by \cref{cor:threshold-rigidity}, is one valid coefficient,
	and every valid $f_k<k+1$ gives a corresponding improvement of
	\cref{thm:main}.  The endpoint $f_k=\rho_k$ cannot itself satisfy the required
	uniqueness, as the second extremal law above shows.  Nevertheless, if
	\cref{conj:optimal-tail} holds, then the hypotheses of
	\cref{cor:parameterized-emc} hold for every $R=\rho_k+\eps$ with $\eps>0$. Thus \cref{conj:optimal-tail} is precisely the missing ingredient for the asymptotically optimal form of our theorem.
	
    \subsection{A numerical improvement}
    The present proof, however, does not provide a route to this conjecture. To see the obstruction, note that \cref{cor:threshold-rigidity} is obtained by scaling $Y_i=Z_i/x$ and setting $T=1/x$.  Its proof passes through \cref{lem:simplex-rigidity}.
	In other words, one would have to replace the hypothesis $T\ge r+1$ in
	\cref{lem:simplex-rigidity} by $T\ge\rho_r$.  This strengthening is false. Indeed, the two configurations corresponding to the cover and clique mechanism satisfy, for $T>r$,
	$K_r(T,0,\ldots,0)=q_{r,T}$ and $K_r(T/r,\ldots,T/r)=\left(\frac rT\right)^r$ respectively.
	They have the same value at $T=\rho_r$.  However, there is a third configuration.
	For
	\[
	y=(1,\ldots,1,T-r+1),
	\quad \text{and }\quad 
	K_r(y)=\frac1{\prod_{i=1}^r y_i}=\frac1{T-r+1}.
	\]
	At $T=\rho_r$ this is strictly larger than the common value of the preceding two configurations:
	\[
	\frac1{\rho_r-r+1}>
	\left(\frac r{\rho_r}\right)^r=q_{r,\rho_r}.
	\]
    Thus the proposed extension of
	\cref{lem:simplex-rigidity} fails precisely near $\rho_r$.

	This failure does not contradict \cref{conj:optimal-tail}. A proof of \cref{conj:optimal-tail} will require a genuinely new ingredient. The conjecture should therefore be viewed as a new problem rather than a routine extension of the method developed here.

    Very recently, Ling proved Samuels' conjecture in full, without requiring the independent summands to be identically distributed \cite{Ling2026}.  For the equal-mean problem relevant here, let $\sigma_k>0$ be determined by
$    \left(1-\frac{1}{k+\sigma_k}\right)^k=\frac{\sigma_k}{1+\sigma_k}.$
    Writing $T=1/x$, Ling's theorem implies the inequality in \cref{conj:optimal-tail} whenever $T\ge k+\sigma_k$, even without the identical-distribution assumption.  A stability analysis of its equality cases further gives some $\varepsilon_k>0$ for which the full conclusion of \cref{conj:optimal-tail}, under the i.i.d.\ assumption, remains valid for $T\ge k+\sigma_k-\varepsilon_k$.  This supplies a smaller coefficient for which the tail bound and its Bernoulli equality characterization hold, so \cref{cor:parameterized-emc} correspondingly lowers the coefficient $k+1$ in \cref{thm:main}.  Nevertheless, the resulting threshold remains strictly larger than $\rho_k$, and hence does not yield the asymptotically sharp range predicted by the Erd\H{o}s matching conjecture. A calculation gives
$	\sigma_5=0.596802\ldots<0.6.$
    Thus, Ling's theorem lowers the coefficient to $k+0.6$ for $k\ge5$.

\section*{Acknowledgement}
The first and third authors gratefully acknowledge the Fourth ECOPRO Student Research
Program, held at the Institute for Basic Science (IBS) in summer 2026,
for its support. The second author would like to thank Ziqing Xiang for informing us that the recent work of Ling~\cite{Ling2026} can be used to obtain the improvement outlined above. The authors acknowledge the use of ChatGPT 5.6 Sol during the
exploratory stage of this project. All mathematical arguments and proofs
presented in the final manuscript were developed and rigorously verified
by the authors. The authors take full responsibility for the content of
the manuscript.


\begin{thebibliography}{99}

		\bibitem{AlonEtAl2012}
		N. Alon, P. Frankl, H. Huang, V. R\"odl, A. Ruci\'nski, and B. Sudakov,
		Large matchings in uniform hypergraphs and the conjectures of Erd\H{o}s and Samuels,
		\emph{J. Combin. Theory Ser. A} \textbf{119} (2012), 1200--1215.

		\bibitem{AlonHuangSudakov2012}
		N. Alon, H. Huang, and B. Sudakov,
		Nonnegative $k$-sums, fractional covers, and probability of small deviations,
		\emph{J. Combin. Theory Ser. B} \textbf{102} (2012), 784--796.
		
		\bibitem{Billingsley1999}
		P. Billingsley,
		\emph{Convergence of Probability Measures}, 2nd ed.,
		Wiley Series in Probability and Statistics, John Wiley \& Sons, New York, 1999.
		
		\bibitem{BollobasDaykinErdos1976}
		B. Bollob\'as, D. E. Daykin, and P. Erd\H{o}s,
		Sets of independent edges of a hypergraph,
		\emph{Quart. J. Math. Oxford Ser. (2)} \textbf{27} (1976), 25--32.
		
		\bibitem{Chung1991}
		F. R. K. Chung,
		Regularity lemmas for hypergraphs and quasi-randomness,
		\emph{Random Structures Algorithms} \textbf{2} (1991), 241--252.
		
		\bibitem{Erdos1965}
		P. Erd\H{o}s,
		A problem on independent $r$-tuples,
		\emph{Ann. Univ. Sci. Budapest. E\"otv\"os Sect. Math.} \textbf{8} (1965), 93--95.
		
		\bibitem{ErdosGallai1959}
		P. Erd\H{o}s and T. Gallai,
		On maximal paths and circuits of graphs,
		\emph{Acta Math. Acad. Sci. Hungar.} \textbf{10} (1959), 337--356.
		
		\bibitem{Frankl2013}
		P. Frankl,
		Improved bounds for Erd\H{o}s' matching conjecture,
		\emph{J. Combin. Theory Ser. A} \textbf{120} (2013), 1068--1072.
		
		\bibitem{Frankl2017Range}
		P. Frankl,
		Proof of the Erd\H{o}s matching conjecture in a new range,
		\emph{Israel J. Math.} \textbf{222} (2017), 421--430.
		
		\bibitem{Frankl2017Triple}
		P. Frankl,
		On the maximum number of edges in a hypergraph with given matching number,
		\emph{Discrete Appl. Math.} \textbf{216} (2017), 562--581.
		
		\bibitem{FranklKupavskii2022}
		P. Frankl and A. Kupavskii,
		The Erd\H{o}s matching conjecture and concentration inequalities,
		\emph{J. Combin. Theory Ser. B} \textbf{157} (2022), 366--400.
		
		\bibitem{FranklLuMaWu2026}
		P. Frankl, H. Lu, J. Ma, and Y. Wu,
		Towards the Erd\H{o}s matching conjecture for $4$-uniform hypergraphs: stability
		and applications,
		arXiv:2602.19230, 2026.
		
		\bibitem{FranklRodl1992}
		P. Frankl and V. R\"odl,
		The uniformity lemma for hypergraphs,
		\emph{Graphs Combin.} \textbf{8} (1992), 309--312.
		
		\bibitem{FranklRodlRucinski2012}
		P. Frankl, V. R\"odl, and A. Ruci\'nski,
		On the maximum number of edges in a triple system not containing a disjoint
		family of a given size,
		\emph{Combin. Probab. Comput.} \textbf{21} (2012), 141--148.
		
		\bibitem{FuEtAl2026}
		W. Fu, Y. Han, G. Wang, J. Yan, P. Zhang, and Z. Zhou,
		Sharp small-deviation inequalities for sums of independent nonnegative random
		variables,
		arXiv:2607.23980v1, 2026.
		
		\bibitem{Gardner2002}
		R. J. Gardner,
		The Brunn--Minkowski inequality,
		\emph{Bull. Amer. Math. Soc. (N.S.)} \textbf{39} (2002), 355--405.
		
		\bibitem{Grunbaum1960}
		B. Gr\"unbaum,
		Partitions of mass-distributions and of convex bodies by hyperplanes,
		\emph{Pacific J. Math.} \textbf{10} (1960), 1257--1261.
		
		\bibitem{HouHuLiu2026}
		J. Hou, C. Hu, and X. Liu,
		A finite-board reduction for the Erd\H{o}s Matching Conjecture and the
		$4$-uniform case via exact certificates,
		arXiv:2605.26060, 2026.
		
		
		\bibitem{KolupaevKupavskii2023}
		D. Kolupaev and A. Kupavskii,
		Erd\H{o}s matching conjecture for almost perfect matchings,
		\emph{Discrete Math.} \textbf{346} (2023), Paper No. 113304.
		
		\bibitem{LuczakMieczkowskaSileikis2017}
		T. \L{}uczak, K. Mieczkowska, and M. \v{S}ileikis,
		On maximal tail probability of sums of nonnegative, independent and identically
		distributed random variables,
		\emph{Statist. Probab. Lett.} \textbf{129} (2017), 12--16.
		
		\bibitem{TanganelliCastrillon2025}
		L. Tanganelli Castrill\'on,
		An improved stability result for Gr\"unbaum's inequality,
		\emph{J. Geom. Anal.} \textbf{35} (2025), Art. No. 248.
		
		\bibitem{VlassisThomas2026}
		N. Vlassis and P. S. Thomas,
		An exact distribution-free test for means of nonnegative random variables,
		arXiv:2607.08415v1, 2026.
		\bibitem{Ling2026}
        Z. Ling,
         On Samuels' conjecture,
        arXiv:2608.18392, 2026.
	\end{thebibliography}
\end{document}